\documentclass[12pt, a4paper]{article}

\PassOptionsToPackage{dvipsnames}{xcolor} 

\usepackage{amsmath, amsthm, amssymb}

\usepackage{bm}                       

\usepackage[headings]{fullpage}    

\usepackage{verse}                 
\usepackage[perpage, marginal]{footmisc}  
\usepackage{footnote}  

\usepackage[shortlabels, inline]{enumitem} 

\usepackage{url}
\usepackage{tikz}
\usepackage{wrapfig}         
\usepackage{epigraph}        
\usepackage{booktabs}        
\usepackage{tocstyle}        
\usepackage{framed}          
\usepackage{verbatim}
\usepackage{aliascnt}        
\usepackage[all]{xy}         

\usepackage[style=numeric, minnames=3,
            doi=false, url=false, isbn=false,
            firstinits=true,
            sortcites=true,
            backend=biber, bibencoding=utf8]{biblatex}      

\usepackage[
  pdftitle={The L2-metric on the space of smooth maps}, 
  pdfauthor={Martins Bruveris},
  colorlinks=true,
  linkcolor=blue,
  urlcolor=blue,
  citecolor=blue,
  hyperfootnotes=false,   
  linktocpage=true]{hyperref}       
\usepackage[all]{hypcap}

\setlist{leftmargin=*,
         itemsep=0.5\itemsep,
         parsep=0.5\parsep,
         topsep=0.5\topsep,
         partopsep=0.5\partopsep}
\setlist[enumerate, 1]{label=(\arabic*), ref=(\arabic*)}

\newlist{proofsteps}{enumerate}{1}
\setlist[proofsteps]{%
  label={\bfseries Step~\arabic{proofstepsi}.\, },
  wide=0pt,
  listparindent=\parindent,
  itemindent=\parindent,
}

\usetocstyle{nopagecolumn}

\author{Martins Bruveris}
\title{The $L^2$-metric on $C^\infty(M,N)$}

\def\signed #1{{\leavevmode\unskip\nobreak\hfil\penalty50\hskip2em
  \hbox{}\nobreak\hfil#1%
  \parfillskip=0pt \finalhyphendemerits=0 \endgraf}}

\newsavebox\mybox

\theoremstyle{plain}

\newtheorem{theorem}{Theorem}[section]

\newaliascnt{proposition}{theorem}
\newtheorem{proposition}[proposition]{Proposition}
\aliascntresetthe{proposition}

\newaliascnt{lemma}{theorem}
\newtheorem{lemma}[lemma]{Lemma}
\aliascntresetthe{lemma}

\newaliascnt{corollary}{theorem}

\aliascntresetthe{corollary}

\theoremstyle{remark}

\newaliascnt{remark}{theorem}

\aliascntresetthe{remark}

\newaliascnt{example}{theorem}

\aliascntresetthe{example}

\newaliascnt{problem}{theorem}

\aliascntresetthe{problem}

\newaliascnt{question}{theorem}
\newtheorem{question}[question]{Question}
\aliascntresetthe{question}

\newaliascnt{openquestion}{theorem}

\aliascntresetthe{openquestion}

\theoremstyle{definition}

\newaliascnt{definition}{theorem}

\aliascntresetthe{definition}

\newaliascnt{project}{theorem}

\aliascntresetthe{project}

\let\oldautoref\autoref
\makeatletter
\renewcommand\autoref[1]{\@first@ref#1,@}
\def\@throw@dot#1.#2@{#1}
\def\@set@refname#1{
    \edef\@tmp{\getrefbykeydefault{#1}{anchor}{}}%
    \def\@refname{\@nameuse{\expandafter\@throw@dot\@tmp.@autorefname}s}%
}
\def\@first@ref#1,#2{%
  \ifx#2@\oldautoref{#1}\let\@secondref\@gobble
  \else%
    \@set@refname{#1}
    \@refname~\ref{#1}
    \let\@secondref\@second@ref
  \fi%
  \@secondref#2%
}
\def\@second@ref#1,#2{%
  \ifx#2@ and~\ref{#1}\let\@nextref\@gobble
  \else, \ref{#1}
    \let\@nextref\@next@ref
  \fi%
  \@nextref#2%
}
\def\@next@ref#1,#2{%
   \ifx#2@, and~\ref{#1}\let\@nextref\@gobble
   \else, \ref{#1}
   \fi%
   \@nextref#2%
}
\makeatother

\let\oldtheequation\theequation
\makeatletter
\def\tagform@#1{\maketag@@@{\ignorespaces#1\unskip\@@italiccorr}}
\renewcommand{\theequation}{(\oldtheequation)}
\makeatother

\def\ep{\varepsilon}

\def\ka{\kappa}

\def\rh{\varrho}

\def\ph{\varphi}

\def\om{\omega}
\def\Ga{\Gamma}

\def\Om{\Omega}

\def\inv{^{-1}}
\def\x{\times}
\def\p{\partial}

\def\exp{\operatorname{exp}}

\let\on=\operatorname

\let\mb=\mathbb
\let\mc=\mathcal
\let\mf=\mathfrak

\newcommand{\N}{{\mathbb{N}}}

\newcommand{\R}{{\mathbb{R}}}

\newcommand{\ud}{\,\mathrm{d}}
\newcommand{\dd}{\mathrm{d}}

\usetikzlibrary{arrows.meta}

\tikzset{
  every node/.style = {font = \footnotesize}
}

\tikzset{
    >={Stealth[length=2mm]}
}

\newcommand{\drawxlabels}[2][fill]{
  \foreach \x/\xtext in {#2} {
    \ifthenelse{\equal{#1}{fill}}
    {
      \draw[shift={(\x,0)}] (0pt,{2.5pt/\scale}) -- (0pt,{-2.5pt/\scale}) 
        node[below,fill=white] {$\xtext$};
    }
    {
      \draw[shift={(\x,0)}] (0pt,{2.5pt/\scale}) -- (0pt,{-2.5pt/\scale}) 
        node[below] {$\xtext$};
    }
  }
}

\newcommand{\drawylabels}[2][fill]{
  \foreach \y/\ytext in {#2} {
    \ifthenelse{\equal{#1}{fill}}
    {
      \draw[shift={(0,\y)}] ({2.5pt/\scale},0pt) -- ({-2.5pt/\scale},0pt) 
        node[left,fill=white] {$\ytext$};
    }
    {
      \draw[shift={(0,\y)}] ({2.5pt/\scale},0pt) -- ({-2.5pt/\scale},0pt) 
        node[left] {$\ytext$};
    }
  }
}

\begin{document}

\maketitle

\begin{abstract}
Let $M$, $N$ be finite-dimensional manifolds with $M$ compact. This paper looks at the Riemnannian geometry on the space $C^\infty(M,N)$ of smooth maps equipped with the $L^2$-Riemannian metric. This metric was used by Ebin and Marsden in the proof of the well-posedness of the incompressible Euler equation and is related to the Wasserstein distance in optimal transport. The paper gives an introduction to the challenges of infinite-dimensional Riemannian geometry and shows how one use general connections to relate the geometry of $N$ and the geometry of $C^\infty(M,N)$.
\end{abstract}


\section{Introduction}

Let $M$ be an orientable compact manifold without boundary of dimension $m$ with volume form $\mu \in \Om^m(M)$ and $(N,g)$ a Riemannian manifold. We assume $N$ is finite-dimensional although it does not have to be compact or complete. The purpose of this note is to discuss the Riemannian geometry of the $L^2$-metric on the space $C^\infty(M,N)$ of smooth maps, defined by
\begin{equation*}
G_q(h, k) = \int_M g_{q(x)}(h(x), k(x)) \,\mu(x)\,.
\end{equation*}
When talking about spaces of maps one has to choose a regularity class and the choice $C^\infty(M,N)$ is one among many. It is not entirely arbitrary, because the space $C^\infty(M,N)$ allows us to use the framework of convenient calculus with its wealth of permitted geometric constructions. Afterwards we will show how to extend the results to Sobolev spaces $H^s(M,N)$ with $s > \dim N / 2$ and to classical $C^k$-spaces $C^k(M,N)$ with $k \in \N$.

\subsection{Applications}
The $L^2$-metric was first used by Ebin and Marsden~\cite{Ebin1970} in the proof of the well-posedness of the incompressible Euler equation. The $L^2$-metric on the whole diffeomorphism group $\on{Diff}(M)$ is not right-invariant, but it is right-invariant when restricted to the subgroup $\on{Diff}_\mu(M)$ of diffeomorphisms preserving the volume form.

A second application is optimal transport. Assume $\mu$ has total volume $1$. We consider the push-forward map
\[
\pi : \on{Diff}(M) \to \on{Vol}_1(M)\,,\quad \ph \mapsto \ph_\ast \mu\,,
\]
from diffeomorphisms onto volume forms with volume $1$. Then $\pi$ is a Riemnnian submersion from $\on{Diff}(M)$ with the $L^2$-metric onto the space of volume forms with a suitable Riemannian metric. More importantly, the geodesic distance of this metric on the space of volume forms is the Wasserstein distance. This was first discovered by Otto~\cite{Otto2001}, a recent overview article is~\cite{Modin2017}.

\subsection{Is this really infinite-dimensional?}

The Riemannian geometry of the $L^2$-metric on $C^\infty(M,N)$ is intimately connected with the Riemannian geometry of $(N,g)$. This can be seen from the following informal derivation of the geodesic equation. The energy of a path $q : [0,1] \to C^\infty(M,N)$ is
\[
E(q) = \frac 12 \int_0^1 \int_M g_{q(t,x)}\big(\p_t q(t,x), \p_t q(t,x)\big) \, \mu(x) \ud t\,.
\]
Using convenient calculus, the following maps can be identified witch each other,
\begin{align*}
q &: (-\ep,\ep) \to C^\infty([0,1], C^\infty(M,N)) \\
q^\wedge &: (-\ep,\ep) \x [0,1] \to C^\infty(M,N) \\
q^{\wedge\wedge} &: (-\ep,\ep) \x [0,1] \x M \to N\,,
\end{align*}
and all three are smooth maps between the respective spaces; this is the exponential law for convenient vector spaces,~\cite[Section~3]{Kriegl1997}. Note that the first map is a one-parameter variation of a smooth path, while the third map is a smooth map between finite-dimensional spaces. We will use $q$ to refer to all three of them. Hence, computing the variation $\p_{\ep}|_{0} E$ of the energy reduces to computing, for each $x \in M$, the variation of the Riemannian energy on $(N,g)$.
\begin{align*}
\p_\ep \big(E(q)\big) &= \frac 12 \int_0^1 \int_M \p_\ep
\left( g_{q(t,x)}\big(\p_t q(t,x), \p_t q(t,x)\big) \right) \, \mu(x) \ud t \\
&= \int_0^1 \int_M
g_{q(t,x)}\big( \nabla_{\p_\ep} \p_t q(t,x), \p_t q(t,x)\big) \, \mu(x) \ud t \\
&= \int_0^1 \int_M
- g_{q(t,x)}\big( \p_\ep q(t,x), \nabla_{\p_t} \p_t q(t,x) \big) \, \mu(x) \ud t\,.
\end{align*}
This shows that $L^2$-geodesics satisfy the equation
\[
\nabla^g_{\p_t} \p_t q(t,x) = 0\,;
\]
in other words, $q(\cdot,x)$ is a geodesic on $(N,g)$ for each $x \in M$.

Why did we call this derivation `informal'? What we did compute is an equation that every $L^2$-geodesic has to satisfy. However, the `geodesic equation' on a Riemannian manifold is the equation
\[
\nabla_{\p_t} \p_t q = 0\,,
\]
with $\nabla$ being its Levi-Civita covariant derivative. The above calculation does not derive a formula for the Levi-Civita covariant derivative of the $L^2$-metric. In fact, because the $L^2$-metric is a strictly \emph{weak} Riemannian metric\footnote{The topology induced by the inner product $G_q(\cdot,\cdot)$ on each tangent space $T_qC^\infty(M,N)$ is \emph{strictly weaker} than the manifold topology.}, the existence of the Levi-Civita covariant derivative is not even guaranteed a priori. To be able to speak of the geodesic equation for the $L^2$-metric, we need to show first that the $L^2$-metric admits a Levi-Civita covariant derivative.

The covariant derivative in infinite dimensions is a tricky object. Formally, if $\mc M$ is an (infinite-dimensional) manifold, it is a map
\[
\nabla: \mf X(\mc M) \x \mf X(\mc M) \to \mf X(\mc M)\,,
\]
and $\nabla_X Y$ is the derivative of the vector field $Y$ in the direction of the vector field $X$. When $\mc M = C^\infty(M,N)$, a vector field is a map
\[
X : C^\infty(M,N) \to C^\infty(M,TN)\,,
\]
subject to $\pi_N \circ X(q) = q$ for all $q \in C^\infty(M,N)$. We don't have a practical way to describe these maps or to work with them. In charts a covariant derivative can be written using Christoffel symbols\footnote{This is true for Levi-Civita covariant derivatives. In general, the relationship between covariant derivatives and Christoffel symbols in infinite dimensions is not entirely clear. It may be that there exist covariant derivatives that do not admit Christoffel symbols although no explicit examples are known.} as
\[
\nabla_X Y(x) = DY(x).X(x) - \Ga_x(X(x),Y(x))\,.
\]
Christoffel symbols, however, can only be defined in charts. On $C^\infty(M,N)$ charts around a map $q$ map an open set in $C^\infty(M,N)$ to an open set in $\Ga(q^\ast TN)$. Charts for $C^\infty(M,N)$ obscure the geometry of $N$. We could also consider charts for $N$, but in general the image of a map $q : M \to N$ will not be contained in a single chart\footnote{One can consider multiple charts covering $M$ and a neighborhood of the image of $q$. In this way one can identify $C^\infty(M,N)$ with a submanifold of a vector space. See \cite{Inci2013} where this construction has been carried out for the spaces $H^s(M,N)$.}. Thus we are forced to conclude that charts are not a helpful tool to relate the geometry of $C^\infty(M,N)$ and that of $N$.

What can we use then? The first tool at our disposal is the description of a covariant derivative $\nabla$ via its \emph{connector} $K : TT\mc M \to T\mc M$,
\[
\nabla_X Y = K \circ TY \circ X\,.
\]
One can view the connector as the coordinate-invariant form of Christoffel symbols, because in local coordinates,
\[
K(x,h;k,l) = (x, l - \Ga_x(k,h))\,.
\]

Our second tool is the observation that if $N \subseteq P$ is a submanifold, then $C^\infty(M,N)$ is a submanifold of $C^\infty(M,P)$. Furthermore, if $(P,g)$ is a Riemannian manifold and $N$ carries the induced Riemannian metric, then the $L^2$-metric on $C^\infty(M,N)$ is the restriction of the $L^2$-metric on $C^\infty(M,P)$. Using Nash's embedding theorem we can embed $(N,g)$ isometrically into some Euclidean space $\R^d$ with the standard metric. Since $C^\infty(M,\R^d)$ is a vector space and the $L^2$-metric reduces to the $L^2$-norm,
\[
G_q(h,k) = \int_{M} \langle h(x), k(x) \rangle \,\mu(x) = \langle h, k \rangle_{L^2(\mu)}\,.
\]
the Riemannian geometry is as simple as one can hope for.

\subsection{Outline}

Our plan for this note is as follows: First we review the differential geometry of the manifold $C^\infty(M,N)$ and the functorial nature of the correspondence
\[
N \rightsquigarrow C^\infty(M,N)\,.
\]
The functorial nature manifests itself in identities such as
\[
TC^\infty(M,N) = C^\infty(M,TN)
\quad\text{and}\quad
\pi_{C^\infty(M,N)}(h) = \pi_{N} \circ h\,,
\]
with $\pi_N : TN \to N$ the canonical projection. 

Next we revisit the extension of the Levi-Civita covariant derivative $\nabla^g$ to vector fields along arbitrary maps. The formula
\[
\nabla^g_X s = K \circ Ts \circ X\,,
\]
can be extended to maps $s : P \to TN$ with $P$ an arbitrary finite- or infinite-dimensional manifold. The choice $P = C^\infty(M, N) \x M$ together with the exponential law alows us to define a covariant derivative $\nabla$ on $C^\infty(M,N)$. We then compute the connector, spray, exponential map and curvature tensor of this connection.

Finally we look at the $L^2$-metric and show that the connection described above is the Levi-Civita connection of the $L^2$-metric. We show this first for the space $C^\infty(M,\R^d)$ and then embed $N$ as an isometric submanifold in $\R^d$.

\subsection{Historical notes}

The $L^2$-metric was first considered on the diffeomorphism group $\on{Diff}(M)$ by Ebin and Marsden in \cite[Section~9]{Ebin1970}. There the authors used the group structure of $\on{Diff}(M)$ and right-invariant vector fields to show that the Levi-Civita covariant derivative exists and computed the spray and the exponential map of the metric. A different proof, also utilizing the group structure, can be found in \cite[Proposition~2]{Bao1993} and a curvature computation in the same setting was performed by Misio{\l}ek in \cite[Proposition~3.4]{Misiolek1993}. Because of their reliance on the group structure these proofs do not extend to the whole space $C^\infty(M,M)$. 

Kainz calculated in \cite{Kainz1984}, generalizing the work of Binz \cite{Binz1980}, the Levi-Civita covariant derivative and the curvature of the related metric
\[
G_q(h,k) = \int_M g_{q(x)}(h(x),k(x)) \, (q^\ast \mu^g)(x)\,,
\]
on the space $\on{Imm}(M,N)$ of immersions. These calculations used the connector-based formulas for the covariant derivative and the curvature that were adopted in this note.

Freed and Groisser summarize in \cite{Freed1989} results about the $L^2$-metric with the proofs only sketched. The $L^2$-metric with a flat ambient space was used in the context of optimal transport by Otto in \cite{Otto2001}. An overview of the $L^2$-metric and its relation to hydrodynamics and optimal transport can be found in the book \cite{Khesin2009} by Khesin and Wendt.

\section{The manifold $C^\infty(M,N)$}
\label{sec:cmn_manifold}

\subsection{Functorial properties}

Let $M$ be compact manifold without boundary and $N$ a finite-dimensional manifold. Then the space $C^\infty(M,N)$ is an infinite-dimensional Fr\'echet manifold. We refer to \cite[Chapter~9]{Kriegl1997} for details on the differentiable structure although the ideas go back to Eells~\cite{Eells1958}. Here want to emphasize the functorial nature of the correspondence
\[
N \rightsquigarrow C^\infty(M,N)\,.
\]
Let $P$ be another finite-dimensional manifold and $f : N \to P$ a smooth map. Then we obtain the smooth map
\[
L_f : C^\infty(M,N) \to C^\infty(M,P)\,,\quad
q \mapsto f \circ q\,,
\]
which is the \emph{left-composition} with $f$. Other notation for $L_f$ that can be found in the literature are $f_\ast$ for \emph{push-forward}, $\om_f$, and $C^\infty(M,f)$ to emphasize that $L_f$ is the transformation of the morphism $f : N \to P$ under the functor $C^\infty(M,\cdot)$.

The functor $C^\infty(M,\cdot)$ commutes with the functors commonly enountered in differential geometry such as the tangent bundle functor. In detail this means that the commutative diagram
\[
\xymatrix{
TN \ar[r]^{Tf} \ar[d]_{\pi_N} & TP \ar[d]^{\pi_P} \\
N \ar[r]_f & P
}
\]
gives rise to the two commutative diagrams---the first obtained by applying $C^\infty(M,\cdot)$ to the above diagram and the second by applying the tangent functor to the map $L_f : C^\infty(M,N) \to C^\infty(M,P)$---shown below:
\[
\xymatrix{
C^\infty(M,TN) \ar[r]^{L_{Tf}} \ar[d]_{L_{\pi_N}} & C^\infty(M,TP) \ar[d]^{L_{\pi_P}} \\
C^\infty(M,N) \ar[r]_{L_f} & C^\infty(M,P)
}
\quad\text{and}\quad
\xymatrix{
TC^\infty(M,N) \ar[r]^{TL_f} \ar[d]_{\pi_{{C^\infty}(M,N)}} & 
TC^\infty(M,P) \ar[d]^{\pi_{{C^\infty}(M,N)}} \\
C^\infty(M,N) \ar[r]_{L_f} & C^\infty(M,P)
}\,.
\]
Because the two functors commute these diagrams coincide. Apart from the identification of the spaces
\[
TC^\infty(M,N) = C^\infty(M,TN)\,,
\]
we also have the following identities summarized in the next lemma.
\begin{lemma}
Let $M$ be compact $N$, $P$ finite-dimensional and $f \in C^\infty(N,P)$. Then
\begin{enumerate}
\item
$\pi_{C^\infty(M,N)} = L_{\pi_N} : C^\infty(M,TN) \to C^\infty(M,N)$
\item
$TL_f = L_{Tf} : C^\infty(M,TN) \to C^\infty(M,TP)$
\end{enumerate}
\end{lemma}

We will write $\pi_{C^\infty}$ for $\pi_{C^\infty(M,N)}$ when the spaces in question are clear. We shall also make use of the following theorem relating submanifolds and the corresponding spaces of maps. A proof can be found in \cite[Proposition~10.8]{Michor1980}.

\begin{proposition}
Let $N \subseteq P$ be a submanifold. Then $C^\infty(M,N)$ is a splitting submanifold of $C^\infty(M,P)$.
\end{proposition}

\subsection{The second tangent bundle}
Let $M$ be a compact manifold and $N$ a finite-dimensional manifold. The two vector bundle structures on $TTN$ induce two corresponding vector bundle structures on $TTC^\infty(M,N) \cong C^\infty(M,TTN)$ which are shown in the corresponding diagram.
\[
\xymatrix@C=-1cm{
& TTC^\infty(M,N) \ar[ld]_{\pi_{TC^\infty}} \ar[rd]^{T\pi_{C^\infty}} & \\
TC^\infty(M,N) \ar[rd]_{\pi_{C^\infty}} & & TC^\infty(M,N) \ar[ld]^{\pi_{C^\infty}} \\
& C^\infty(M,N) &
}
\qquad
\xymatrix@C=-1cm{
& C^\infty(M,TTN) \ar[ld]_{L_{\pi_{TN}}} \ar[rd]^{L_{T\pi_{M}}} & \\
C^\infty(M,TN) \ar[rd]_{L_{T\pi_M}} & & C^\infty(M,TN) \ar[ld]^{L_{\pi_{M}}} \\
& C^\infty(M,N) &
}
\]
The vertical bundle,
\begin{align*}
VTC^\infty(M,N) &= \ker T\pi_{C^\infty} = \ker L_{T\pi_N} \\
&= \{ \xi \in C^\infty(M,TTN) \,:\, \xi(x) \in VTN\; \forall x \in M \} \\
&= C^\infty(M,VTN) \,,
\end{align*}
consists of maps that map pointwise into the vertical bundle $VTN$. Since
\[
TC^\infty(M,N) \x_{C^\infty(M,N)} TC^\infty(M,N)
\cong
C^\infty(M, TN \x_N TN)\,,
\]
the vertical lift also acts pointwise,
\[
\on{vl}_{C^\infty} : C^\infty(M,TN \x_N TN) \to C^\infty(TTN)\,,\quad
\on{vl}_{C^\infty}(h,k) = \on{vl}_M \circ (h,k)\,,
\]
as does the vertical projection
\[
\on{vpr}_{C^\infty} : C^\infty(M,VTN) \to C^\infty(M,TN)\,,\quad
\on{vpr}_{C^\infty}(\xi) = \on{vpr}_M \circ \xi\,.
\]
The statement about the vertical lift follows directly from the definition,
\[
\on{vl}_{C^\infty}(h,k)(x) = \frac{\dd}{\dd t}\bigg|_{t=0} \left(h + tk \right) (x)
= \frac{\dd}{\dd t}\bigg|_{t=0} \left(h(x) + tk(x) \right) = \on{vl}_M(h(x),k(x))\,,
\]
and the statement about the vertical projection, because
\[
\on{vpr}_{C^\infty} = L_{\on{pr_2}} \circ \left(\on{vl}_{C^\infty}\right)\inv
= L_{\on{pr}_2} \circ L_{\on{vl}_M}\inv = L_{\on{pr}_2 \circ \on{vl}_M\inv} = L_{\on{vpr}_M}\,.
\]
The canonical flip on $N$, $\ka_{N} : TTN \to TTN$ is given in coordinates by
\[
\ka_N (x,h;k,l) = (x,k;h,l)\,,
\]
and can be characterized as the unique smooth map $TTN \to TTN$ satisfying the equation $\p_t \p_s c(t,s) = \ka_N \circ \p_s \p_t c(t,s)$ for each $c \in C^\infty(\R^2,N)$. Using this property we see that the canonical flip on $C^\infty(M,N)$ also acts pointwise,
\[
\ka_{C^\infty} : C^\infty(M,TTN) \to C^\infty(M,TTN)\,,\quad
\ka_{C^\infty}(\xi) = \ka_N \circ \xi\,.
\]

\section{A covariant derivative on $C^\infty(M,N)$}
\label{ss:cmn_covariant_derivative}

Let $\nabla^g$ be the Levi-Civita covariant derivative of $(N,g)$ and $K : TTN \to TN$ its connector, such that
\[
\nabla^g_X Y = K \circ TY \circ X : N \to TN \to TTN \to TN\,,
\]
for $X,Y \in \mf X(N)$. In local coordinates $K(x,h;k,l) = (x,l - \Ga_x(k,h))$. See~\cite[22.8]{Michor2008b} for details. 

The connector point of view allows us to extend the covariant derivative to act on vector fields along maps. If $f : P \to N$ is a smooth map, $s : P \to TN$ a vector field along $f$, meaning $\pi_N \circ s = f$ and $X \in \mf X(P)$ is a vector field on $N$, then we can define
\[
\nabla^g_X s = K \circ Ts \circ X : P \to TP \to TTN \to TN \,.
\]

There is no reason, why $P$ has to be a finite-dimensional manifold and so we can apply this construction with $P = C^\infty(M,N) \x M$ to obtain a covariant derivative on $C^\infty(M,N)$. This construction follows~\cite[Section~4.2]{Bauer2011b}. Let $Q$ be an arbitrary manifold, finite- or infinite-dimensional. We identify
\[
s \in C^\infty(Q, C^\infty(M,TN))
\qquad\text{and}\qquad
X \in \mf X(Q)
\]
with
\[
s^\wedge \in C^\infty(Q \x M,TN)
\qquad\text{and}\qquad
(X, 0_M) \in \mf X(Q \x M)\,.
\]
Then we can define the covariant derivative
\[
\nabla^g_{(X,0_M)} s^\wedge \in C^\infty(Q \x M, TN)
\]
as above and thus we set
\[
\nabla_X s = \left( \nabla^g_{(X,0_M)} s^\wedge\right)^\vee \in C^\infty(Q, C^\infty(M,TN))\,.
\]
When $Q=C^\infty(M,N)$, this defines a covariant derivative on $C^\infty(M,N)$. Next we compute the connector, the spray and the curvature of this connection.

\begin{lemma}
Let $\nabla$ be the covariant derivative on $C^\infty(M,N)$ induced by $\nabla^g$. Then the connector $K$ of $\nabla$ is given by
\[
K : C^\infty(M,TTN) \to C^\infty(M,TN)\,,\quad K(\xi) = K^g \circ \xi\,,
\]
where $K^g : TTN \to TN$ is the connector of $\nabla^g$. Furthermore, the linear connection $C$ of $\nabla$ is
\[
C : C^\infty(M, TN \x_N TN) \to C^\infty(M,TTN)\,,\quad C(h,k) = C^g \circ (h,k)\,,
\]
with $C^g : TN \x_N TN \to TTN$ the linear connection of $\nabla^g$.
\end{lemma}

\begin{proof}
We will show that for all $s \in C^\infty(Q,C^\infty(M,TN))$ and all $X \in \mf X(Q)$, we have
\[
\nabla_X s = L_{K^g} \circ Ts \circ X\,.
\]
To see this note that using \autoref{lem:Lf_wedge_identity} and \autoref{lem:Ts_wedge_identity} we have
\begin{align*}
\left( \nabla_X s \right)^\wedge &= \nabla^g_{(X,0_M)} s^\wedge \\
&= K^g \circ Ts^\wedge \circ (X, 0_M) \\
&= K^g \circ \left(Ts \circ X\right)^\wedge \\
&= \left(L_{K^g} \circ Ts \circ X \right)^\wedge \,.
\end{align*}

To prove that $C$ is the linear connector of $\nabla$, we will show that $C$ defined via $C = L_{C^g}$ satisfies the identity
\[
K(\xi) = \on{vpr}_{C^\infty}\big( \xi - C(T\pi_{C^\infty}.\xi, \pi_{TC^\infty}(\xi))\big)\,,
\]
for all $\xi \in C^\infty(TTN)$ and hence is the linear connector of $\nabla$. First note that
\begin{align*}
T\pi_{C^\infty} 
&= T\pi_{C^\infty(M,N)} = TL_{\pi_N} = L_{T\pi_N} \,,\quad\text{and} \\
\pi_{TC^\infty}
&= \pi_{C^\infty(M,TN)} = L_{\pi_{TN}}\,.
\end{align*}
Then
\begin{multline*}
\on{vpr}_{C^\infty}\big(\xi - C(T\pi_{C^\infty}.\xi, \pi_{TC^\infty}(\xi))\big)(x) = \\
= \on{vpr}_{M}\big(\xi(x) - C^g(T\pi_N.\xi(x), \pi_{TN}(\xi(x))) \big)
= K^g(\xi(x)) = K(\xi)(x)\,.
\end{multline*}
Hence $C = L_{C^g}$ is the linear connector of $\nabla$.
\end{proof}

The proof relied on the following two lemmas

\begin{lemma}
\label{lem:Lf_wedge_identity}
Let $s \in C^\infty(P, C^\infty(M, N))$ and $f \in C^\infty(N,Q)$. Then
\[
L_f \circ s = \left(f \circ s^\wedge\right)^\vee\,.
\]
\end{lemma}

\begin{proof}
Let $x \in M$ and $y \in P$. Then we have
\begin{align*}
\left(L_f \circ s\right)^\wedge(y,x) &= \left(L_f \circ s\right)(y)(x)
= L_f(s(y))(x) \\
&= \left(f \circ s(y)\right)(x)
= f(s(y)(x)) = f\left(s^\wedge(y,x)\right)
= \left(f \circ s^\wedge\right)(y,x)\,.\qedhere
\end{align*}
\end{proof}

\begin{lemma}
\label{lem:Ts_wedge_identity}
Let $s \in C^\infty(P, C^\infty(M,N))$ and $X \in \mf X(P)$. Then
\[
Ts \circ X = \left( Ts^\wedge \circ (X,0_M) \right)^\vee\,.
\]
\end{lemma}

\begin{proof}
Take $x \in M$ and $y \in P$. Then, differentiating $s^\wedge(y,x) = s(y)(x)$ we obtain
\[
T_{(y,x)}s^\wedge.(k,h) = T_x s(y).h + \left(T_y s.k\right)(x)\,,
\]
with $h \in T_xM$ and $k \in T_y P$. When we set $h = 0_x$, this becomes 
\[
T_{(y,x)}s^\wedge.(k,0_x) = \left(T_y s.k\right)(x)\,,
\]
which can be written in terms of a vector field $X \in \mf X(P)$ as
\[
\big(T s^\wedge \circ (X, 0_M)\big)(y,x) = \big(Ts.X(y)\big)(x) = \left(Ts.X\right)^\wedge(y,x)\,.\qedhere
\]
\end{proof}

\begin{proposition}
Let $\nabla$ be the covariant derivative on $C^\infty(M,N)$ induced by $\nabla^g$.
\begin{enumerate}
\item
\label{thm:ebin1970:spray}
Let $\Xi^g : TN \to TTN$ be the geodesic spray of $(N,g)$. Then
\[
\Xi : C^\infty(M,TN) \to C^\infty(M,TTN)\,,\quad
X \mapsto \Xi^g \circ X\,,
\]
is the geodesic spray $\Xi$ of $\nabla$ and it is a $C^\infty$-mapping.
\item
\label{thm:ebin1970:exp}
Let $\exp^g : TN \supseteq U \to N$ be the exponential map on $(N,g)$, defined on a neighbourhood $U$ of the zero-section, and $\mc U = C^\infty(M,U) \subseteq C^\infty(M,TN)$. Then
\[
\exp : \mc U \to C^\infty(M,N)\,,\quad h \mapsto \exp^g \circ h\,,
\]
is the exponential map of $\nabla$ and $\exp$ is a $C^\infty$-mapping.
\end{enumerate}
\end{proposition}

\begin{proof}
The geodesic spray of a connection can be written in terms of the linear connection as
\[
\Xi(h) = C(h,h)\,,
\]
and since for $\nabla$ we have $C(h,h) = C^g \circ (h,h)$, it follows that
\[
\Xi(h) = C(h,h) = C^g \circ (h,h) = \Xi^g \circ h\,,
\]
thus showing~\ref{thm:ebin1970:exp}.

Using~\ref{thm:ebin1970:spray} a curve $q : (-\ep,\ep) \to C^\infty(M,N)$ is a geodesic if for all $x \in M$
\begin{align*}
\p_t^2 q(t,x) &= \Xi(\p_t q(t))(x) \\
&= \Xi^g(\p_t q(t,x))\,,
\end{align*}
in other words, if $t \mapsto q(t,x)$ is a geodesic in $N$. Thus, for a tangent vector $h \in C^\infty(M,U)$ at $q$ we have
\[
\left(\on{exp}_q h\right)(x) = \exp^{g}_{q(x)} h(x)\,.
\]
This shows~\ref{thm:ebin1970:exp}.
\end{proof}

\begin{proposition}
Let $\nabla$ be the covariant derivative on $C^\infty(M,N)$ induced by $\nabla^g$ and $R^g : TN \x TN \x TN \to TN$ the curvature tensor of $(N,g)$. Then
\[
R : C^\infty(M,TN) \x C^\infty(M,TN) \x C^\infty(M,TN) \to C^\infty(M,TN)\,,\quad
(X,Y,Z) \mapsto R^g \circ (X,Y,Z)\,,
\]
is the curvature tensor of $\nabla$.
\end{proposition}

\begin{proof}
Let $X,Y,Z$ be vector fields on $C^\infty(M,N)$, that is $X: C^\infty(M,N) \to C^\infty(M,TN)$ with $L_{\pi_N} \circ X = \on{Id}_{C^\infty}$ and similarly for $Y,Z$.
Using the definition of the curvature and the connection $\nabla$ we have
\begin{align*}
R(X,Y,Z) &=
\nabla_X \nabla_Y Z - \nabla_Y \nabla_X Z - \nabla_{[X,Y]} Z \\
&= \left( \nabla_{(X,0_M)} \nabla_{(Y,0_M)} Z^\wedge
- \nabla_{(Y,0_M)} \nabla_{(X,0_M)} Z^\wedge
- \nabla_{([X,Y],0_M)}Z^\wedge \right)^\vee\,,
\end{align*}
where
\begin{align*}
Z^\wedge &\in C^\infty( C^\infty(M,N) \x M, TN )\,, &
(X, 0_M),\, (Y,0_M) \in \mf X(C^\infty(M,N) \x M)\,.
\end{align*}
Using
\[
([X,Y],0_M) = \left[ (X,0_M), (Y,0_M) \right]\,,
\]
and \cite[Section~24.5]{Michor2008b}, we obtain
\begin{align*}
R \circ (X,Y,Z) &= \left(R^g \circ \left(Tf \circ (X,0_M), Tf \circ (Y, 0_M), Z^\wedge \right)\right)^\vee\,,
\end{align*}
where $f = \pi_N \circ Z^\wedge$. Now, using~\autoref{lem:Lf_wedge_identity},
\[
f^\vee = \left(\pi_N \circ Z^\wedge\right)^\vee = L_{\pi_N} \circ Z = \on{Id}_{C^\infty}\,,
\]
and hence using~\autoref{lem:Ts_wedge_identity},
\[
Tf \circ (X,0_M) = \left( Tf^\vee \circ X \right)^\wedge 
= \left(T\on{Id}_{C^\infty} \circ X\right)^\wedge = X^\wedge\,.
\]
Now we apply \autoref{lem:Lf_wedge_identity} once more and obtain
\begin{align*}
R \circ (X,Y,Z)
&= \left( R^g \circ \left( X^\wedge, Y^\wedge, Z^\wedge \right) \right)^\vee
= L_{R^g} \circ (X,Y,Z)\,,
\end{align*}
as required.
\end{proof}

\section{The $L^2$-metric}

\subsection{Definition and covariant derivative}

Finally we arrive at the $L^2$-metric on $C^\infty(M,N)$. The metric is defined by
\begin{equation}
\label{eq:l2_metric}
G_q(h, k) = \int_M g_{q(x)}(h(x), k(x)) \,\mu(x)\,.
\end{equation}
The following theorem summarizes the properties of $G$.

\begin{theorem}
\label{thm:ebin1970}
Let $G$ be the $L^2$-metric on $C^\infty(M,N)$ defined by \eqref{eq:l2_metric}.
\begin{enumerate}
\item
$G$ defines a smooth weak Riemannian metric on $C^\infty(M,N)$.
\item
The Levi-Civita covariant derivative of $G$ coincides with the covariant derivative described in \autoref{ss:cmn_covariant_derivative}.
\end{enumerate}
\end{theorem}

\begin{proof} 
\begin{proofsteps}
\item
The $L^2$-metric is smooth because smooth curves are mapped to smooth curves. This is one of the main principles of convenient calculus.

\item
Let $\nabla$ be the covariant derivative from \autoref{ss:cmn_covariant_derivative}. To identify the Levi-Civita covariant derivative, we first consider the case $N=\R^d$, not necesserily with the Euclidean metric. Then $C^\infty(M,\R^d)$ is a vector space and we can use coordinate formulas for the Christoffel symbols. Note that the identity
\[
D_{q,m} G_\cdot(h,k) = \int_M D_{q(x),m(x)} g_\cdot(h(x),k(x)) \,\mu(x)\,,
\]
shows that the directional derivative of the metric $G$ is given by the integral over the pointwise directional derivatives of the finite-dimensional metric $g$. Therefore
\begin{align*}
 \frac 12 \big( D_{q,m} G_\cdot(h,k) - D_{q,h} G_\cdot(k,m) - D_{q,k} G_\cdot(m,h)\big)
 &= \int_M g_{q(x)}\left(\Ga^g_{q(x)}(h(x),k(x)),m(x) \right) \,\mu(x)\,,
\end{align*}
which shows that the Christoffel symbols of $G$ exist and are given pointwise by those of $g$,
\[
\Ga^G_q(h,k)(x) = \Ga^g_{q(x)}(h(x),k(x))\,,
\]
and hence we can identify the connector $K^G$ of the Levi-Civita covariant derivative $\nabla^G$,
\[
K^G(\xi)(x) = K^g(\xi(x)) = K(\xi)(x)\,,
\]
where $K$ is the connector of $\nabla$. Thus $\nabla^G = \nabla$ when $N = \R^d$.
\item
We embed $(N,g)$ isometrically as a submanifold of $(\R^d,\bar g)$ for some $d$, where $\bar g$ is the standard Riemannian metric. Then $C^\infty(M,N)$ with the $L^2$-metric is an isometric submanifold of $C^\infty(M,\R^d)$ with the $L^2$-metric. Let $P^g : T\R^d|_N \to TN$ be the orthogonal projection. Then
\[
TC^\infty(M,\R^d)|_{C^\infty(M,N)} \cong C^\infty(M,T\R^d|_N)
\]
and the $G$-orthogonal projection exists and is given by
\[
P : C^\infty(M,T\R^d|_N) \to C^\infty(M,TN)\,,\quad
P(h) = P^g \circ h\,.
\]
Hence the Levi-Civita covariant derivative on $C^\infty(M,N)$ exists and its connector $K$ is given by
\[
K = P \circ \bar K(\xi) = P^g \circ K^{\bar g} \circ \xi = K^g \circ \xi\,,
\]
where $\bar K$ is the connector of $\nabla$ on $C^\infty(M,\R^d)$ and $K^g$ is the connector of $(N,g)$.\qedhere
\end{proofsteps}
\end{proof}

\subsection{Reparametrization invariance}

The $L^2$-metric is \emph{not} invariant with respect to the whole group $\on{Diff}(M)$. In fact, if $\ph \in \on{Diff}(M)$, then
\begin{align*}
G_{q\circ \ph}(h\circ \ph,k\circ \ph) &= \int_M g_{q\circ \ph}(h\circ \ph,k\circ \ph)\, \mu
= \int_M \ph^\ast \left(g_q(h,k)\, \ph_\ast\mu\right)
= \int_M g_{q} (h, k)\, \ph_\ast \mu\,.
\end{align*}
Thus $\ph$ leaves $G$ invariant if and only if $\ph_\ast \mu = \mu$.
Note that in local coordinates,
\[
\left(\ph_\ast \mu\right) = \ph_\ast\left(\rh\, dx^1 \wedge \dots \wedge dx^n\right)
= \rh \circ \ph\inv \left(\det D\ph\inv\right) \, dx^1 \wedge \dots \wedge dx^n\,.
\]
Let
\[
\on{Diff}_\mu(M) = \left\{ \ph \in \on{Diff}(M) \,:\, \ph^\ast \mu = \mu \right\}\,,
\]
be the subgroup of diffeomorphisms preserving $\mu$. Then the $L^2$-metric $G$ is right-invariant with respect to $\on{Diff}_\mu(M)$. However, we obtain invariance with respect to the whole group for several objects associated to the metric.

\begin{proposition}
\label{prop:l2_invariance}
Let $G$ be the $L^2$-metric on $C^\infty(M,N)$ defined by \eqref{eq:l2_metric}. Then the following are $\on{Diff}(M)$-equivariant:
\begin{enumerate}
\item
The connector and the linear connection of the Levi-Civita metric.
\item
The geodesic spray and the exponential map.
\item
The curvature tensor.
\end{enumerate}
\end{proposition}

\begin{proof}
The connector $K$ of the $L^2$-metric is given by $K(\xi) = K^g \circ \xi$, where $K^g$ is the connector of the metric $(N,g)$. Then we have for $\ph \in \on{Diff}(M)$,
\[
K(\xi \circ \ph) = K^g \circ (\xi \circ \ph) = \big( K^g \circ \xi \big) \circ \ph
= K(\xi) \circ \ph\,.
\]
Hence the connector is $\on{Diff}(M)$-equivariant. The proof for the other maps is the same.
\end{proof}

Another way to view this theorem is by observing none of these objects, which are derived from the Levi-Civita covariant derivative, depend on the volume form $\mu$. In fact, the covariant derivative was defined in \autoref{ss:cmn_covariant_derivative} before we introduced the $L^2$-metric itself. This means that the Levi-Civita covariant derivative $\nabla$ remains unchanged if we change the volume form $\mu$. Since the $L^2$-metric is invariant under $\on{Diff}_\mu(M)$, we obtain equivariance of the Levi-Civita covariant derivative under all diffeomorphisms that preserve \emph{some} volume form. Thus leads to the following question:

\begin{question}
Does every element $\ph \in \on{Diff}(M)$ leave some volume form invariant?
\end{question}

If the answer to this question is `yes', then we can prove \autoref{prop:l2_invariance} without having to resort to explicit formulas for the connector and the other maps.

\subsection{Other spaces of maps}

In place of the space $C^\infty(M,N)$ we could also consider the space $H^s(M,N)$ of maps of a given Sobolev regularity $s > \dim M / 2$ or the space $C^k(M,N)$ of maps with a finite number of derivatives. The $L^2$-metric,
\begin{equation*}
G_q(h, k) = \int_M g_{q(x)}(h(x), k(x)) \,\mu(x)\,,
\end{equation*}
extends smoothly to these spaces. The smoothness of $G$ follows from the fact that composition from the left with smooth functions is smooth, i.e., the map $q \mapsto g_q$, and that pointwise multiplication is a smooth bilinear map. 

To calculate the Levi-Civita covariant derivative, the spray and the curvature we could carefully redo the proofs from the previous sections on the larger spaces. Following this path one would encounter some difficulties, because a vector field $X$ on $H^s(M,N)$ now is a map
\[
X \in C^\infty\big( H^s(M,N), H^s(M,TN) \big)
\]
and $X^\wedge : H^s(M,N) \x M \to TN$ is now a map with mixed regularity: $C^\infty(M,N)$ in the first component and $H^s$ in the second. Identifying the precise regularity class of $X^\wedge$ is a nontrivial task. This has been done in some cases \cite{Michor2013b}. Here we shall follow a different path.

First, some good news. The correspondences
\[
N \rightsquigarrow H^s(M,N)
\qquad\text{and}\qquad
N \rightsquigarrow C^k(M,N)
\]
are also functorial in nature and everything in \autoref{sec:cmn_manifold} remains valid for both $H^s(M,N)$ and $C^k(M,N)$. This follows directly from the construction of charts on these spaces and because left-composition with smooth maps is a smooth map.

\begin{proposition}
\label{prop:l2_hs}
Let $G$ be the $L^2$-metric on $H^s(M,N)$ with $s > \dim M / 2$ defined by~\eqref{eq:l2_metric}.
\begin{enumerate}
\item
$G$ defines a smooth weak Riemannian metric on $H^s(M,N)$.
\item
The Levi-Civita covariant derivative $\nabla$ of $G$ exists and its connector is given by
\[
K : H^s(M,TTN) \to H^s(M,TN)\,,\quad K(\xi) = K^g \circ \xi\,.
\]
\item
The geodesic spray is given by
\[
\Xi : H^s(M,TN) \to H^s(M,TTN)\,,\quad \Xi(h) = \Xi^g \circ h\,.
\]
\item
The exponential map is defined on a neighborhood $\mc U \subseteq H^s(M,TN)$ of the zero-section and given by
\[
\on{exp} : \mc U \to H^s(M,N)\,,\quad \on{exp}(h) = \on{exp}^g \circ h\,.
\]
\item
The curvature tensor is given by
\[
R : H^s(M,TN) \x H^s(M,TN) \x H^s(M,TN) \to H^s(M,TN)\,,\quad
(X,Y,Z) \mapsto R^g \circ (X,Y,Z)\,,
\]
\end{enumerate}
\end{proposition}

\begin{proof}
\begin{proofsteps}
\item
The map $K$, given by $K(\xi) = K^g \circ \xi = L_{K^g}(\xi)$ is well-defined and smooth. Following~\cite{Michor2008b} a smooth map $K : TT\mc M \to T\mc M$ is the connector of a connection if it satisfies the following three properties:
\begin{enumerate}
\item
$K \circ \on{vl}_{M} = \on{pr}_2 : TM \x_M TM \to TM$, where $\on{vl}_M$ is the vertical lift.
\item
$K$is linear for the first vector bundle structure on $TTM$, $\pi_{TM} : TTM \to TM$.
\item
$K$ is linear for the second vector bundle structure on $TTM$, $T\pi_M : TTM \to TM$.
\end{enumerate}
To be symmetric $K$ additionally has to satisfy $K \circ \ka_M = K$. Because we have already seen that $K$ satisfies these identities when restricted to $C^\infty(M,TTN)$ and because $C^\infty(M,TTN)$ is dense in $H^s(M,TTN)$ it follows that $K$ is the connector of a symmetric connection $\nabla$ on $H^s(M,N)$. 

It remains to show that $\nabla$ is compatible with the Riemannian metric. The compatibility condition that a Levi-Civita connection satisfies is,
\[
X\big( G(Y,Z) \big) = G( \nabla_X Y, Z ) + G( Y, \nabla_X Z)\,,
\]
with vector fields $X,Y,Z$ and it can be expressed in terms of the connector as
\begin{align*}
\on{pr}_2 \circ TG \circ (TY, TZ) \circ X &= G(K \circ TY \circ X, Z) + G(Y,K \circ TZ \circ X) \\
&= G \circ \left( K \x \on{Id}_{TM} + \on{Id}_{TM} \x K \right) \circ (TY, TZ) \circ X\,;
\end{align*}
here $\on{pr}_2 : T\R \cong \R \x \R \to \R$ is the projection onto the second component. Thus the compatibility condition for the connector is
\begin{align*}
\on{pr}_2 \circ TG 
= G \circ \left( K \x \on{Id}_{TM} + \on{Id}_{TM} \x K \right)\,.
\end{align*}
Again, using the density argument we can conclude that $\nabla$ is the Levi-Civita connection of the $L^2$-metric on $H^s(M,N)$.

\item
The geodesic spray is given by $\Xi(h) = C(h,h)$, where $C$ is the linear connection of the Levi-Civita derivative. The linear connection $C = L_{C^g}$ is characterized by the equation
\[
K(\xi) = \on{vpr}_{C^\infty}\big( \xi - C(T\pi_{C^\infty}.\xi, \pi_{TC^\infty}(\xi))\big)\,,
\]
and we can use the density argument to show that this identity extends from $C^\infty(M,TTN)$ to $H^s(M,TTN)$. This proves the formula for the geodesic spray. That $\on{exp} = L_{\on{exp}^g}$ is the $L^2$-exponential map follows by direct verification.

\item
For the curvature tensor we also use the density argument, this time in coordinate charts. In local coordinates the curvature is given by
\[
R_q(h,k,l) = \Ga_q(h,\Ga_q(k,l)) - \Ga_q(k, \Ga_q(h,l)) + d\Ga(h,l)(q).k - d\Ga(k,l)(q).h\,,
\]
and we have already shown that the Christoffel symbols on $C^\infty(M,N)$ and $H^s(M,N)$ coincide. Thus we obtain that the identity $R = L_{R^g}$ for the curvature tensor on $H^s(M,TN)$ is valid for tangent vectors in $C^\infty(M,TN)$ and hence by density for all tangent vectors. \qedhere
\end{proofsteps}
\end{proof}

The proof of \autoref{prop:l2_hs} can be repeated verbatim with the spaces $H^s(M,N)$ replaced by $C^k(M,N)$. By doing so we obtain the following proposition.

\begin{proposition}
The statements of \autoref{prop:l2_hs} also hold for the $L^2$-metric on the spaces $C^k(M,N)$ with $k \in \mb N_0$.
\end{proposition}


\subsection*{Acknowledgements}

I would like to thank Peter Michor and Stephen Marsland for helpful discussions and their gentle encouragement. This work benefited from the support and hospitality of the Isaac Newton Institute for Mathematical Sciences during the programme `Growth form and self-organisation' and as such this work was supported by: EPSRC grant number EP/K032208/1.

\addcontentsline{toc}{chapter}{Bibliography}
\printbibliography

\end{document}